\documentclass[hidelinks,onefignum,onetabnum]{siamart250211}

\usepackage{lipsum}
\usepackage{amsfonts}
\usepackage{graphicx}
\usepackage{epstopdf}
\usepackage{algorithmic}

\usepackage{float}
\usepackage{subfig}
\usepackage{cleveref}
\usepackage{bm}
\usepackage{tabularx}
\usepackage{color}
\usepackage{multirow}

\usepackage{algorithm}
\usepackage{threeparttable}
\usepackage{listings}
\usepackage{threeparttable}
\usepackage{listings}
\ifpdf
  \DeclareGraphicsExtensions{.eps,.pdf,.png,.jpg}
\else
  \DeclareGraphicsExtensions{.eps}
\fi

\newsiamremark{remark}{Remark}
\newsiamremark{assumption}{Assumption}
\newsiamremark{hypothesis}{Hypothesis}
\crefname{hypothesis}{Hypothesis}{Hypotheses}
\newsiamthm{claim}{Claim}
\newsiamremark{fact}{Fact}
\crefname{fact}{Fact}{Facts}

\title{Efficient Solution of Generalized Sylvester Equations via Preconditioned Alternating Anderson Acceleration \thanks{Submitted to the editors DATE.
\funding{This work was funded by the National Natural Science Foundation of China under grant Nos. 12571398, 12361080, 12401497, Jiangxi Provincial Natural Science Foundation (No. 20232BAB211007) and National Key R \&D Program of China (No. 2023YFA1011303).}}}

\def\T{{\rm T}}
\def\F{{\rm F}}
\def\g{\gamma}

\def\hC{\widehat{C}}
\def\hN{\widehat{N}}
\def\hM{\widehat{M}}

\def\tg{\widetilde{\gamma}}

\def\calL{\mathcal{L}}
\def\calP{\mathcal{P}}
\def\calR{\mathcal{R}}
\def\calF{\mathcal{F}}

\def\hcalL{\widehat{\mathcal{L}}}
\def\hPi{\widehat{\Pi}}
\def\bbR{\mathbb{R}}

\def\vec{\text{vec}}

  \author{Hongjia Chen\thanks{Department of Mathematics, Nanchang University, 999 Xuefu Road, Nanchang 330031, China (\email{chenhongjia@ncu.edu.cn}).}
\and Chun-Hua Zhang\thanks{School of Mathematics and Information Science, Nanchang Hangkong
University, Nanchang 330063, China (\email{chzringlang@163.com}).}
\and Zhongming Teng\thanks{College of Computer and Information Science, Fujian Agriculture and Forestry University, Fuzhou 350002, China
(\email{peter979@163.com}).}
\and Lei Du\thanks{School of Mathematical Sciences, Dalian University of Technology, Dalian 116024, China (\email{dulei@dlut.edu.cn}).
}}

\usepackage{amsopn}

\begin{document}

\maketitle

\begin{abstract}
This paper considers the numerical solution of generalized Sylvester matrix equations, which arise in many scientific and engineering applications but remain challenging to solve efficiently, particularly when the coefficient matrices are general and the spectral radius of the associated operator is large but not greater than $1$. We propose a new iterative method, termed preconditioned-alternating Anderson acceleration (P-aAA), which combines a matrix-oriented variant of Anderson acceleration (AA) with a novel preconditioning strategy. The method alternates between preconditioned fixed-point iterations and Anderson acceleration updates, thereby reducing both computational cost and iteration count. A key contribution is the development of an efficient preconditioning operator based on a first-order Neumann series approximation, which avoids expensive operator inversions while enhancing convergence. We theoretically prove that the proposed preconditioning operator accelerates the convergence rate without increasing the overall computational complexity. Extensive numerical experiments further demonstrate that the proposed approach consistently outperforms existing state-of-the-art methods for both medium- and large-scale problems, achieving substantial reductions in computation time and iteration number.
\end{abstract}



\section{Introduction}
\label{sec:intro}
In this paper, we consider the numerical solution of generalized Sylvester matrix equations of the form,
\begin{equation}
A X + X B^{\T} + \mathop{\sum}\limits_{i=1}^{m} N_{i} X M_{i}^{\T}  = C, \label{eq:gensylv}
\end{equation}
or, equivalently,
\begin{equation}
\calL(X) - \Pi(X) = C, \label{eq:oper_sylv}
\end{equation}
where the matrix Sylvester operator 
\begin{equation*}
\calL(X):= AX + XB^{\T}, 
\end{equation*}
and the matrix operator
\begin{equation*}
\Pi(X):= - \mathop{\sum}\limits_{i=1}^{m} N_{i} X M_{i}^{\T}.
\end{equation*}
In the equations above, $X\in\bbR^{n\times n}$, $A,B\in\bbR^{n\times n}$ are the nonsingular matrices, $C, N_i,M_i\in\bbR^{n\times n}$ for $i=1,2,\dots,m$.

We assume the spectral $A$ and $-B$ are disjoint \cite{sim2016}, which implies 
that the operator $\calL$ is invertible. Moreover, we assume that $\rho(\calL^{-1}\Pi)<1$, where $\rho$ denotes the spectral radius. 
The assumption on the spectral radius implies that \eqref{eq:gensylv} has a unique solution. 

The generalized Sylvester equation finds broad applications in scientific computing and control theory. For instance, considering a special case of the generalized Lyapunov equation \eqref{eq:gensylv}, where $B=A$, $M_i=N_i$. This type of equation arises primarily from model order reduction of bilinear and stochastic systems \cite{benner2011,damm2008}. Additionally, problems such as the discretization of partial differential equations can also be formulated as generalized Sylvester equations; see references \cite{palitta2016,sim2016}.

Theoretically, solving a generalized Sylvester matrix equation \eqref{eq:gensylv} is equivalent to solving a linear system of size
$n^2 \times n^2$,
\begin{equation}
\mathcal{A}\,\vec(X) = 
\left( B\otimes I_n + I_n \otimes A  + \mathop{\sum}\limits_{i=1}^{m} M_i\otimes N_i\right )\vec(X) = \vec(C),
\label{eq:linear}
\end{equation}
where the symbol $\otimes$ denotes the Kronecker product, and $\vec(\cdot)$ converts a matrix into a column vector by stacking its columns. If $\mathcal{A}\in\mathbb{R}^{n^2 \times n^2}$ is invertible, the equation \eqref{eq:linear} has a unique solution $X$. 
When the dimension $n$ is small, the linear system \eqref{eq:linear} can be solved by direct methods, such as Gaussian elimination. However, for very large $n$, direct methods become impractical due to their high computational complexity of $O(n^6)$ and memory storage requirement of $O(n^4)$, making them unsuitable for large-scale problems. In such cases, iterative algorithms for solving the linear system can be employed, as discussed in references \cite{demmel1997,saad2003}. 
Over the past two decades, another class of methods has been developed that circumvents the use of the Kronecker product formulation. These methods construct iterative schemes for solving \eqref{eq:gensylv} directly by leveraging the efficient, matrix-oriented implementation of standard vector iterative methods. For example, the Bilinear ADI (BilADI) method \cite{benner2013} extends the low-rank ADI algorithm \cite{benner2009} from standard Lyapunov equations to generalized Lyapunov equations. A greedy low-rank approach was introduced in \cite{kressner2015} to solve the generalized Lyapunov equations. 
One well-established class of methods for solving matrix equations defined by large, sparse matrices is based on projection and is commonly referred to as projection methods \cite{breiten2019,jarlebring2018,shank2016}. 
Specifically, a rational Krylov subspace approach for large-scale generalized Lyapunov equations was introduced in \cite{breiten2019}. For solving generalized Lyapunov and Sylvester equations, low-rank commuting Krylov subspace (LRCK) method has also been developed, as detailed in \cite{jarlebring2018}. 
The generalized Lyapunov extended Krylov (GLEK) method, proposed in \cite{shank2016}, is a non-stationary iterative algorithm that incorporates extended Krylov subspace method to solve large-scale generalized Lyapunov equations.

More recently, various techniques for solving linear matrix equations have been studied, such as preconditioning \cite{voet2025}, Riemannian optimization \cite{bioli2025}, the truncated conjugate gradient (CG) method \cite{sim2023}, a subspace-CG method \cite{palitta2025}, and the Sherman–Morrison–Woodbury (SMW) formula \cite{hao2021}.

Efficiently solving generalized Sylvester equations remains a highly challenging problem in numerical linear algebra. While many classical algorithms have been developed for generalized Lyapunov equations, research on generalized Sylvester equations has mainly focused on special cases, such as those that are symmetric positive definite, or where the matrices $N_i$ and $M_i$ are very sparse and low-rank. In contrast, relatively few studies have addressed more general forms of the generalized Sylvester equation, for example, cases in which $N_i$ and $M_i$ are arbitrary matrices and the spectral radius of the associated operator $\rho(\calL^{-1}\Pi)$ is slightly less than $1$.

In this paper, we propose a novel alternating acceleration framework for solving generalized Sylvester matrix equations. The core idea of the proposed algorithm is to alternately apply Anderson acceleration (AA) and a newly developed preconditioning technique to improve both convergence rate and computational efficiency. We first derive a matrix-oriented version of Anderson acceleration. Subsequently, we introduce a new preconditioning operator $\mathcal{P}^{-1}$, which is applied to the residual
\[
\calR(X) = \calL(X) - \Pi(X) - C, 
\]
that is, $\mathcal{P}^{-1}\mathcal{R}(X)$. The preconditioning operator $\mathcal{P}^{-1}$ is constructed using a first-order truncation of the Neumann series expansion of $(\calL - \Pi)^{-1}$, providing an accurate and computationally inexpensive approximation of the inverse operator. This design accelerates convergence without increasing the overall computational cost. Finally, by combining Anderson acceleration with the preconditioning operator, we develop an alternating iteration scheme, referred to as the preconditioned-alternating Anderson Acceleration (P-aAA) method, which leverages the complementary strengths of preconditioning and acceleration.

The remainder of this paper is organized as follows. In Section \ref{sec:AA}, we briefly introduce the Anderson acceleration (AA) method and then propose its matrix-oriented variant. In Section 3, we present the main contribution of this work—the development of the P-aAA method. This novel approach integrates a newly constructed preconditioning technique with a matrix-oriented form of Anderson acceleration. We provide a theoretical proof showing that the proposed preconditioning accelerates convergence while incurring no additional computational cost in practice. 
In Section 4, numerical experiments demonstrate that our algorithm exhibits strong competitiveness compared with state-of-the-art methods, achieving improvements in both computational time and iteration count. Concluding remarks are given in Section 5.

\textbf{Notation}. $\bbR^{n\times n}$ is the set of $n$-by-$n$ real matrices, $\bbR^{n} = \bbR^{n\times 1}$ and 
$\bbR = \bbR^{1}$. $I_n \in\bbR^{n\times n}$ is the $n$-by-$n$ identity matrix. Given a matrix/vector $B$, $B^{\T}$
denotes its transpose respectively. $\|B\|_2$ and $\|B\|_{\F}$ are the 2-norm and Frobenius norm of $B$, respectively. Other notations will be explained at their first appearances.

\section{Anderson acceleration and its variant}\label{sec:AA}
In this section, we briefly review the Anderson acceleration (AA) method and then propose its matrix-oriented variant for splitting-type iterative schemes applied to generalized Sylvester matrix equations. The AA method, originally introduced by Anderson \cite{anderson1965}, was designed to accelerate the convergence of general fixed-point iterations for solving systems of nonlinear equations. It can be effectively applied to both linear and nonlinear problems. The basic idea of AA is to combine information from several previous iterates to generate an improved approximation. Under certain conditions, when all past iteration information is utilized, AA applied to linear systems can be shown to generate the same iterates as the GMRES method; see \cite{potra2013,walker2011}.
Recently, extensive research has been devoted to the convergence analysis and theoretical properties of AA \cite{potra2013,sterck2021,toth2015,walker2011}, and comprehensive reviews can be found in \cite{anderson2019,saad2025}. Furthermore, it has been successfully applied across a variety of fields, including computational mathematics and scientific computing \cite{an2017,ganine2013,mai2020,wei2021}.

\subsection{Anderson acceleration}
Given a function $g:\bbR^{n}\rightarrow\bbR^{n}$, 
consider the equation
\begin{equation*}
x = g(x).
\end{equation*} 
A solution $x^{*}$ satisfying $g(x^{*}) = x^{*}$ is called a fixed point of $g$. Here, $g$ is referred to as the fixed point function. A classical approach to solve the problem is to employ a fixed-point iteration scheme, that is, given
an initial guess $x_0$ for the solution, to compute the sequence
\begin{equation}
x_{i+1} = g(x_i), \quad  i=0,1,\dots,k \label{eq:fix}
\end{equation}
until some convergence criterion is met. 

Define the residual $f(x) = g(x) - x$, and let $f_i = f(x_i)$ and $g_i = g(x_i)$, where $x_i$ denotes
the sequence of iterates generated by \eqref{eq:fix}.
Define the differences
\begin{align*}
\Delta x_i &= x_{i+1} - x_i,\\ 
\Delta f_i &= f_{i+1} - f_i, \\
\Delta g_i &= g_{i+1}- g_i,
\end{align*} and form the matrices 
\begin{align*}
\mathcal{X}_k &= (\Delta x_{k-m_k},\dots,\Delta x_{k-1}), \\
\mathcal{F}_k &= (\Delta f_{k-m_k},\dots,\Delta f_{k-1}), \\
\mathcal{G}_k &= (\Delta g_{k-m_k},\dots,\Delta g_{k-1}),
\end{align*}
for $i = 0,1,\dots, k-1$.
The general form of Anderson acceleration is presented in \Cref{alg:AA}.

\begin{algorithm}[t]
\caption{Anderson acceleration \label{alg:AA}}
\begin{algorithmic}[1]
\STATE Given $x_0$ and $m\geq 1$.
\STATE Set $x_1 = g(x_0)$.
\FOR {$k=1,2,\dots$}
\STATE Set $m_k = \min\{m,k\}$.
\STATE Set $f_k = g(x_k) - x_k$ and compute $\mathcal{X}_k,\mathcal{F}_k$, and $\mathcal{G}_k$.
\STATE Determine $\gamma^{(k)} = (\gamma_0^{(k)},\dots,\gamma_{m_k -1}^{(k)})^{\T}$ by solving
              \begin{equation*}
              \mathop{\min}\limits_{\gamma = (\gamma_0,\dots,\gamma_{m_{k} - 1})^{\T}} \|f_k - \mathcal{F}_k \gamma\|_2.
              \end{equation*}
\STATE Set $x_{k+1} = (1-\beta_k)(x_k - \mathcal{X}_k\gamma^{(k)}) + \beta_k(g(x_k) - \mathcal{G}_k\gamma^{(k)})$.

\ENDFOR
\end{algorithmic}
\end{algorithm}

In \Cref{alg:AA}, $m_k$ denotes the Anderson depth and is bounded above by a maximum depth $m$. The parameter $\beta_k \in (0,1]$, referred to as the damping parameter, is typically chosen heuristically.

\subsection{The matrix-oriented Anderson acceleration}
Now, we start by writing the matrix-oriented Anderson acceleration. Following \eqref{eq:oper_sylv}, given an initial $X_0 = 0$, we consider the fixed-point iteration of the form
\begin{equation*}
\calL(X_{k+1}) =  \Pi(X_k) +  C, \quad k = 0,1,2,\dots.
\end{equation*}
Then we have 
\begin{equation}
X_{k+1} = \calL^{-1}(\Pi(X_k)) + \calL^{-1}(C). \label{eq:Xk}
\end{equation}
Solving \eqref{eq:Xk} can be equivalently reformulated as solving a standard Sylvester matrix equation,
\begin{align}
A X_{k+1} + X_{k+1} B^{\T} = - \mathop{\sum}\limits_{i=1}^{m} N_{i} X_k M_{i}^{\T}  + C.\label{eq:standSylv}
\end{align} 
The Bartels–Stewart (BS) algorithm \cite{bar1972} can be applied to solve \eqref{eq:standSylv} for $X_{k+1}$. 
Since the BS algorithm requires the Schur decomposition of matrices $A$ and $B$, i.e.,
\begin{equation*}
A = Q_A H_A Q_A^{\T},\quad B = Q_B H_B Q_B^{\T}, 
\end{equation*}
where $Q_A$ and $Q_B$ are $n\times n$ orthogonal matrices, and $H_A$ and $H_B$ are $n\times n$ real block-triangular matrices. We first transform \eqref{eq:gensylv} into the following form to avoid repeated Schur decompositions of $A$ and $B$ at each iteration,
\begin{equation}
H_AY + Y H_B^{\T} + \mathop{\sum}\limits_{i=1}^m \hN_i Y {\hM}_i^{\T} = \hC, \label{eq:Hsylv}
\end{equation} 
where $\hC = Q_A^{\T}CQ_B$, $\hN_i = Q_A^{\T}N_iQ_A$, and $\hM_i = Q_B^{\T}M_iQ_B$. 

We define the operators
\begin{equation*} 
\hcalL(Y) :=  H_A Y + Y H_B^{\T}
\end{equation*}
and 
\begin{equation*}
\hPi(Y) :=  - \mathop{\sum}\limits_{i=1}^{m} \hN_{i} Y \hM_{i}^{\T}.
\end{equation*}
Similar to \eqref{eq:Xk}, the fixed-point iteration of equation \eqref{eq:Hsylv} is given by 
\begin{equation}
Y_{k+1} = G(Y_k) = \hcalL^{-1}(\hPi(Y_k)) + \hcalL^{-1}(\hC). \label{eq:Yk}
\end{equation}
Let $G_k =  G(Y_k)$ and residual $F_k = G_k - Y_k$. 
Set the difference
\begin{align*}
\Delta Y_k &= Y_{k+1} - Y_k,\\
\Delta F_k &= F_{k+1} - F_{k},\\
\Delta G_k &= G_{k+1} - G_k,
\end{align*}
and obtain
\begin{align*}
\mathcal{Y}_k &= (\Delta Y_{k-m_k},\dots,\Delta Y_{k-1})\in\bbR^{n\times (n\cdot m_k)}, \\
\mathcal{F}_k &= (\Delta F_{k-m_k},\dots,\Delta F_{k-1})\in\bbR^{n\times (n\cdot m_k)}, \\
\mathcal{G}_k &= (\Delta G_{k-m_k},\dots,\Delta G_{k-1})\in\bbR^{n\times (n\cdot m_k)}. 
\end{align*}
Now, the least-square problem in matrix-oriented AA algorithm is rewritten as
\begin{align}
\mathop{\min}\limits_{\gamma^{(k)}}\| F_k - \calF_k \gamma^{(k)} \|_{\F} &= 
\mathop{\min}\limits_{\gamma^{(k)}}\| G(Y_k) - Y_k - \calF_k \gamma^{(k)} \|_{\F}  \nonumber\\
&= \mathop{\min}\limits_{\gamma^{(k)}}\| Y_{k+1} - Y_k - \calF_k \gamma^{(k)} \|_{\F} =\mathop{\min}\limits_{\gamma^{(k)}}\| \Delta Y_k - \calF_k \gamma^{(k)} \|_{\F}. \label{eq:least}
\end{align}
The least-square solution of \eqref{eq:least} is then 
\begin{equation*}
\gamma^{(k)} = \calF_k^{\dagger}\Delta Y_k,
\end{equation*}
where ${\dagger}$ is the Moore–Penrose inverse. 
Indeed, if $\Delta Y_k$ is low-rank and $m_k >1$, the matrix $\calF_k$ may become ill-conditioned during the iteration. To avoid this issue, we employ a truncated singular value decomposition (SVD) to obtain a low-rank approximation of $\calF_k$, that is,
\begin{align*}
\calF_k \approx U_r \Sigma_r V_r^{\T}, 
\end{align*}
where $U_r\in\bbR^{n\times r}$, $\Sigma_r\in\bbR^{r\times r}$, and $V_r\in\bbR^{m_k\cdot n \times r}$, with $r\ll n$
denoting the truncating rank. Therefore, the approximation solution $\tg^{(k)}$ of $\g^{(k)}$ is computed by
\begin{equation*}
\tg^{(k)} = V_r\Sigma_r^{-1}U_r^{\T}\Delta Y_k. 
\end{equation*}
The approximation of $Y_k$ is computed by
\begin{equation*} 
Y_{k+1} = (1-\beta_k)(Y_k - \mathcal{Y}_k\tg^{(k)}) + \beta_k(G(Y_k) - \mathcal{G}_k\tg^{(k)}).
\end{equation*}
In this paper, we set $\beta_k = 1$, then
\begin{equation}
Y_{k+1} = G(Y_k) - \mathcal{G}_k\tg^{(k)}.\label{eq:Yk+1}
\end{equation}
Finally, $X_{k+1}$ is computed by $X_{k+1} = Q_A Y_{k+1}Q_B^{\T}$.

In some applications, it may be advantageous to delay the initiation of Anderson acceleration until the underlying fixed-point iteration has been applied for a certain number of initial steps. This strategy, referred to as delayed Anderson acceleration, was introduced in \cite{walker2011a}. The complete matrix-oriented Anderson acceleration procedure is summarized in \Cref{alg:matAA}.

\subsection{Stopping criterion}
To complete the matrix-oriented AA algorithm for the efficient solution of \eqref{eq:gensylv}, it remains to specify the stopping criterion. For the fixed-point iteration \eqref{eq:Yk}, one needs to compute the relative residual 
\begin{equation*}
RRes = \frac{\|\calR(Y_k)\|_{\F}}{\|\hC\|_{\F}} = \frac{\|\hcalL(Y_{k}) -  \hPi(Y_k) -  \hC\|_{\F}}{\|\hC\|_{\F}}.
\end{equation*}
To reduce the computational cost of the residual $\calR(Y_k)$, we utilize the relation derived from \eqref{eq:Yk}:
\begin{equation}
F_k = G(Y_k) - Y_k = Y_{k+1} -Y_k = \hcalL^{-1}(\hPi(Y_k)) + \hcalL^{-1}(\hC) - Y_k. \label{eq:Fk}
\end{equation}
Then we derive the following relation from \eqref{eq:Fk}:
\begin{equation}
\hcalL(Y_{k+1} - Y_k) = - (\hcalL(Y_k) - \hPi(Y_k) - \hC) = - \calR(Y_k). \label{eq:L(dX)}
\end{equation}

From \eqref{eq:L(dX)}, the residual $\calR(Y_k)$ can be approximated by $\hcalL(Y_{k+1} - Y_k)$, since the latter can be evaluated with great computational efficient. Specifically, computing $\hcalL(Y_{k+1} - Y_k)$ requires only $4n^3+2n^2$ flops, whereas directly evaluating $\calR(Y_k)$ incurs $(4m+4) n^3 + (m+2)n^2$ flops due to the summation over $m$ terms in its formulation. This approximation reduces the per-iteration cost significantly, particularly when $m$ is large, while still providing a reliable estimate of the residual for iterative methods.

Based on the above analysis, we define the stopping criterion for the matrix-oriented Anderson acceleration as follows:
\begin{equation*}
RRes = \frac{\|\hcalL(Y_{k+1} - Y_k)\|_\F}{\|\hC\|_\F} = \frac{\|H_A(Y_{k+1} - Y_{k}) - (Y_{k+1} - Y_{k}) H_B^{\T})\|_\F}{\|\hC\|_\F}.
\end{equation*}

\begin{algorithm}[h!]
\caption{The matrix-oriented Anderson acceleration}\label{alg:matAA}
\begin{algorithmic}[1]
\REQUIRE Given $A,B,N_i,M_i, C\in\mathbb{R}^{n\times n}$ for $i = 1,2\dots,m$. A stopping criterion $\epsilon$ and number of stored residual differences $\Delta F_k$: $m_{AA}=0$. Acceleration begins once the iteration count equals $AA_{start}$ (if $AA_{start} > 0$). Maximum number of stored residual differences $\Delta F_k$: $m_{Max}$ 
\ENSURE  $X_{k+1}$ is an approximation to the solution of \eqref{eq:gensylv}.
\vspace{2pt}\hrule\vspace{2pt}
\STATE Set $Y_0 = 0$.
\STATE Set $\Delta F_k = [ ~]$ and $\Delta G_k = [~ ]$.
\STATE Compute the Schur decompositions $A = Q_AH_AQ_A^{\T}$ and $B = Q_BH_BQ_B^{\T}$.
\STATE Compute $\hC= Q_A^{\T}CQ_B$, $\hN_i = Q_A^{\T}N_iQ_A$, and $\hM_i = Q_B^{\T}M_iQ_B$.
\STATE Solve $H_AY_1 + Y_1 H_B^{\T} = -\mathop{\sum}\limits_{i=1}^m \hN_i Y_0 \hM_i^{\T} + \hC$ by
             Bartels-Stewart method \cite{bar1972}. 
\FOR {$k=1,2,\dots,$}
\STATE Compute $g_{cur} = G(Y_k)$ by solving 
             $H_A Y_{k+1} + Y_{k+1}H_B^{\T} = -\mathop{\sum}\limits_{i=1}^m \hN_i Y_{k} \hM_i^{\T} + \hC$.
             \vspace{-6pt}
\STATE Compute $f_{cur} = G(Y_k) - Y_k = Y_{k+1} - Y_k$.
\IF {$RRes< \epsilon$.}
\STATE break 
\ENDIF
\IF {$m_{Max} = 0$ or $k<AA_{start}$}
\STATE $Y_{k+1} = g_{cur}$.
\ELSE
\IF{$k>AA_{start}$}
\IF{$m_{AA} < m_{Max}$}
\STATE $\Delta F_k = [\Delta F_k \quad  f_{cur} - f_{old}]$.
\STATE $\Delta G_k = [\Delta G_k \quad  g_{cur} - g_{old}]$.
\STATE $m_{AA} = m_{AA} +1$.
\ELSE  
\STATE $\Delta F_k = [\Delta F_k(:,n+1:m_{AA}\cdot n) \quad  f_{cur} - f_{old}]$.
\STATE $\Delta G_k = [\Delta G_k(:,n+1:m_{AA}\cdot n) \quad  g_{cur} - g_{old}]$.
\ENDIF
\ENDIF
\ENDIF
\STATE Set $f_{old} = f_{cur}$ and $g_{old} = g_{cur}$.
\IF{$m_{AA} = 0$}
\STATE Set $Y_{k+1} = g_{cur}$.
\ELSE
\STATE Solve the least-square problem $\mathop{\min}\limits_{\g^{(k)}} \|f_{cur} - \Delta F_k \g^{(k)}\|_{\F}$.
\STATE Set $Y_{k+1} = G(Y_k) - \Delta G_k\g^{(k)}$.
\ENDIF
\ENDFOR
\RETURN $X_{k+1} = Q_A Y_{k+1}Q_B^{\T}$.
\end{algorithmic}
\end{algorithm}

\section{A preconditioned alternating Anderson acceleration}
In this section, we propose a novel preconditioning technique and integrate it with the Anderson acceleration, leading to the proposed P-aAA method. 
\subsection{Preconditioning}
In the AA algorithm, the convergence rate of the fixed-point iteration $Y_{k+1} = G(Y_k)$ 
is influenced by the magnitude of the spectral radius. To accelerate this fixed-point iteration, we design a preconditioning operator $\calP^{-1}:\bbR^{n\times n}\rightarrow \bbR^{n\times n}$ applied to the residual $\calR(Y_k)$, that is 
\begin{equation}
\calP^{-1}(\calR(Y_k)) = \calP^{-1}(\hcalL(Y_{k}) -  \hPi(Y_k) -  \hC).
\end{equation}
The preconditioned iteration scheme is as follows:
\begin{equation}
Y_{k+1} = Y_k - \calP^{-1}(\calR(Y_k)). \label{eq:pre_iter}
\end{equation}
We now describe how to construct the preconditioning operator $\calP^{-1}$, such that the preconditioned operator is close to the identity operator $\mathcal{I}$. That is,
\begin{equation*}
\calP^{-1}(\hcalL - \hPi) \approx \mathcal{I}.
\end{equation*}
A natural approach is to define $\calP^{-1} = (\hcalL - \hPi)^{-1}$.
Consequently, applying the preconditioning operator $\calP^{-1}$ to the residual $\calR(Y_k)$ is equivalent to solving the following generalized Sylvester equation for $\Delta P$
\begin{equation}
H_A \Delta P + \Delta P H_B^{\T} + \mathop{\sum}\limits_{i=1}^{m} \hN_{i} \Delta P \hM_{i}^{\T} = \calR(Y_k),\quad k = 0,1,2,\dots. \label{eq:dY}
\end{equation}
Solving such a generalized Sylvester equation generally requires iterative methods. A common way is to reformulate it via a fixed-point iteration into the following form:
\begin{equation}
H_A \Delta P_{j+1} + \Delta P_{j+1} H_B^{\T} = - \mathop{\sum}\limits_{i=1}^{m} \hN_{i} \Delta  P_j \hM_{i}^{\T} + \calR(Y_k), \quad j = 0,1,2,\dots.\label{eq:dYk}
\end{equation}

To solve \eqref{eq:dYk}, each iteration step requires solving a standard Sylvester equation by means of the BS algorithm.
However, the convergence of the fixed-point iteration scheme is strongly influenced by the spectral radius $\rho(\hcalL^{-1}\hPi)$. Moreover, since the computational complexity of the BS algorithm is $O(n^3)$, the computation time increases significantly after multiple iterations. Therefore, we consider replacing $\calP^{-1} = (\hcalL - \hPi)^{-1}$ with an approximate operator of 
$\calP^{-1}$ to avoid the increased computational cost resulting from repeatedly solving \eqref{eq:dY} with the BS algorithm during the iteration. 

Since $\rho(\hcalL^{-1}\hPi) < 1$,  the preconditioning $\calP^{-1}$ can be expressed via the convergent Neumann series:
\begin{align*}
\mathcal{P}^{-1} = (\hcalL - \hPi)^{-1} &= (\mathcal{I} - \hcalL^{-1}\hPi)^{-1}\hcalL^{-1} = \sum_{k=0}^{\infty} (\hcalL^{-1} \hPi)^k \hcalL^{-1}.
\end{align*}
Hence,
\begin{align*}
\mathcal{P}^{-1}(\calR(Y_k))  =\,  & \hcalL^{-1}(\calR(Y_k))  + \hcalL^{-1}\Big(\hPi\big(\hcalL^{-1}(\calR(Y_k))\big)\Big)
\\ &+ \hcalL^{-1}\hPi\left(\hcalL^{-1}\Big(\hPi\big(\hcalL^{-1}(\calR(Y_k))\big)\Big)\right)  +  \cdots .
\end{align*}

A first-order truncated expansion of the operator $\calP^{-1}$ is given by
\begin{equation}
\widetilde{\mathcal{P}}^{-1}(\calR(Y_k)) = \hcalL^{-1}(\calR(Y_k)) + \hcalL^{-1}\left(\hPi\left(\hcalL^{-1}(\calR(Y_k))\right)\right),\label{eq:secondNeu}
\end{equation}
which we employ as the truncated preconditioner in place of the full operator $\calP^{-1}$.

In practice, the computation of the truncated preconditioner $\widetilde{\calP}^{-1}(\calR(Y_k))$ consists of two parts. First, the Sylvester operator $\hcalL$ is applied to the residual $\calR(Y_k)$ to obtain $P_1 = \hcalL^{-1}(\calR(Y_k))$, where $P_1$ satisfies to the matrix equation
\begin{equation}
H_A P_1 + P_1H_B^{\T} = \calR(Y_k). \label{eq:P1}
\end{equation}
Note that in the AA algorithm, equation  \eqref{eq:P1}  is equivalent to the equation \eqref{eq:L(dX)} defining the update in the iteration. Hence, the difference $Y_{k} - Y_{k+1}$ satisfies \eqref{eq:P1} and can be regarded as its solution.

In the second part, the Sylvester operator $\hcalL$ is applied to $\hPi(\hcalL^{-1}(R(Y_k)))$, which requires solving the Sylvester equation:
\begin{equation}
H_A P_2 + P_2 H_B^{\T} = -\mathop{\sum}\limits_{i=1}^m \hN_i P_1 \hM_i^{\T}. \label{eq:P2}
\end{equation}

Equation 
\eqref{eq:P2} can be solved using the BS algorithm.

Finally, the approximate solution for $\Delta P$ in \eqref{eq:dY} is computed by $\Delta \widetilde{P} = P_1 + P_2$. According 
to \eqref{eq:pre_iter}, the update is computed as $Y_{k+1} = Y_k - \Delta \widetilde{P}$.

The procedure of the preconditioning is presented in \Cref{alg:pred}.

\begin{algorithm}[h!]
\caption{The approximate preconditioner \label{alg:pred}}
\begin{algorithmic}[1]
\REQUIRE Real block-triangular matrices $H_A, H_B\in\bbR^{n\times n}$, $\hN_i, \hM_i\in\bbR^{n\times n}$, the residual $\calR(Y_k)$, the difference $Y_{k} - Y_{k+1}$.
\ENSURE  The preconditioning $\Delta \widetilde{P}$
\STATE Set $P_1 = Y_{k} - Y_{k+1}$. 
\STATE Compute $P_2$ by solving the Sylvester matrix equation \eqref{eq:P2}.
\STATE The preconditioning $\Delta \widetilde{P} = P_1 + P_2$.
\end{algorithmic}
\end{algorithm}

\subsection{The complete algorithm}
In this section, we combine the preconditioning technique with the Anderson acceleration method and propose a P-aAA algorithm. The P-aAA algorithm is an alternating scheme, in which computations are carried out alternately according to formulas \eqref{eq:Yk+1} and \eqref{eq:pre_iter}.

In the P-aAA algorithm, we first perform the Schur decomposition of matrices $A$ and $B$ 
in \eqref{eq:gensylv}  derive the iterative scheme \eqref{eq:Hsylv}. Incorporating delayed Anderson acceleration, we compute $Y_{k+1}$ using the iterative scheme \eqref{eq:pre_iter} from step $1$ to step $k$ when $k<AA_{start}$.
Then, when $k>AA_{start}$ and $\text{mod}(k,2)=1$, it utilizes Anderson acceleration for the iteration; 
otherwise, it continues to use the preconditioning technique to update $Y_{k+1}$.

We now provide a detailed description of how AA is employed to solve equation \eqref{eq:Hsylv} at the $k$-th step when $k>AA_{start}$ and $\text{mod}(k,2)=1$. We compute $Y_{k+1}$ by solving the equation
\begin{equation*}
Y_{k+1} = Y_{k} - \calP^{-1}(\calR(Y_{k})),
\end{equation*}
where $Y_{k}\in\bbR^{n\times n}$ is obtained from the $k$-th iteration. 
In the P-aAA algorithm, to enhance memory efficiency, we store only the most recent difference 
\begin{align*}
\calF_{k} = \Delta F_{k-1} &= F_{k} - F_{k-1} \\
& = G(Y_{k}) - Y_{k} - G(Y_{k-1}) + Y_{k-1}\\
& = (Y_{k+1} - Y_{k}) - (Y_{k} - Y_{k-1}) = \Delta Y_{k} - \Delta Y_{k-1},
\end{align*}
where $\calF_{k} \in\bbR^{n\times n}$.
This strategy also helps reduce the dimension of the least-squares problem in the Anderson acceleration procedure.
According to \eqref{eq:dY}, the least-square problem in the P-aAA algorithm is given by
\begin{align}
\mathop{\min}\limits_{\g^{(k)}}\| F_{k} - \calF_{k} \g^{(k)} \|_{\F} &= 
\mathop{\min}\limits_{\g^{(k)}}\| G(Y_{k}) - Y_{k} - \calF_{k} \g^{(k)} \|_{\F}  \nonumber\\
&= \mathop{\min}\limits_{\g^{(k)}}\| Y_{k+1} - Y_{k} - \calF_{k} \g^{(k)} \|_{\F} \nonumber\\
& =\mathop{\min}\limits_{\g^{(k)}}\| \Delta Y_{k} - (\Delta Y_{k} - \Delta Y_{k-1}) \gamma^{(k)} \|_{\F}. \label{eq:leastPAAA}
\end{align}
The approximate solution $\tg^{(k)}$ of  \eqref{eq:leastPAAA} follows the same least-squares problem as in the AA algorithm; therefore, the detailed steps are omitted here.

Finally, the iteration of P-aAA can be given by
\begin{equation}
Y_{k+1} = 
\begin{cases}
Y_{k} - \calP^{-1}(\calR(Y_{k})), & \text{if} ~~ k < AA_{start}; \\
G(Y_k) - \mathcal{G}_k\tg^{(k)}, & \text{if} ~~ k > AA_{start}~ \text{and}~ \text{mod}(k,2) = 1.
\end{cases}\label{eq:PAAA}
\end{equation}
The procedure of the P-aAA algorithm is presented in \Cref{alg:pred_AAA}. 
\begin{remark}
In equation \eqref{eq:PAAA}, if we replace $Y_{k} - \calP^{-1}(\calR(Y_{k}))$ with \eqref{eq:Yk}, we refer to the resulting 
iterative scheme as the alternating Anderson acceleration (aAA) method.
\end{remark}

\begin{algorithm}[h!]
\caption{The prenconditioned alternating Anderson acceleration}\label{alg:pred_AAA}
\begin{algorithmic}[1]
\REQUIRE Given $A,B,N_i,M_i\in\mathbb{R}^{n\times n}$ for $i = 1,2\dots,m$ and $C\in\mathbb{R}^{n\times n}$. A stopping criterion $\epsilon$ and number of stored residual differences $\Delta F_k$: $m_{AA}=0$. Acceleration begins once the iteration count equals $AA_{start}$ (if $AA_{start} > 0$). 
\ENSURE  $X_{k+1}$ is an approximation to the solution of \eqref{eq:gensylv}.
\vspace{2pt}\hrule\vspace{2pt}
\STATE Set $Y_0 = 0$.
\STATE Compute the Schur decompositions $A = Q_AH_AQ_A^{\T}$ and $B = Q_BH_BQ_B^{\T}$.
\STATE Compute $\hC = Q_A^{\T}CQ_B$, $\hN_i = Q_A^{\T}N_iQ_A$, and $\hM_i = Q_B^{\T}M_iQ_B$.
\STATE 
\vspace{-1pt}
Solve $H_AY_1 + Y_1 H_B^{\T} = -\mathop{\sum}\limits_{i=1}^m \hN_i Y_0 \hM_i^{\T} + \hC$ by
             Bartels-Stewart method \cite{bar1972}. \vspace{-5pt}
\FOR {$k=1,2,\dots,$}
\STATE \vspace{-1pt}
Compute $g_{cur} = G(Y_k)$ by solving 
             $H_A Y_{k+1} + Y_{k+1}H_B^{\T} = -\mathop{\sum}\limits_{i=1}^m \hN_i Y_{k} \hM_i^{\T} + \hC$.
\STATE \vspace{-5pt} Compute the residual $f_{cur} = G(Y_k) - Y_k = Y_{k+1} - Y_{k}$.
\IF {$RRes< \epsilon$.}
\STATE break 
\ENDIF
\IF {$k <AA_{start}$}
\STATE Set  $Y_{k+1} = Y_k - \Delta \widetilde P$, $\Delta \widetilde P$ is computed by \Cref{alg:pred}.
\ELSE
\IF {$k > AA_{start}$}
\STATE $\Delta F_k = f_{cur} - f_{old}$;
\STATE $\Delta G_k = g_{cur} - g_{old}$;
\STATE $m_{AA} = 1$;
\ENDIF
\ENDIF
\IF {$k = AA_{start}$ or $mod(k,2) =1$} 
\STATE $f_{old} = f_{cur}$; 
\STATE $g_{old} = g_{cur}$;
\STATE $m_{AA} = 0$;
\ENDIF 
\IF {$m_{AA} = 0$}
\STATE $Y_{k+1} = Y_k -  \Delta \widetilde P$, $\Delta \widetilde P$ is computed by \Cref{alg:pred}.
\ELSE 
\STATE Solve the least-square problem $\mathop{\min}\limits_{\g^{(k)}} \|f_{cur} - \Delta F_k \g^{(k)}\|_{\F}$.
\STATE Compute $Y_{k+1} = g_{cur} - \Delta G_k \g^{(k)}$
\ENDIF
\ENDFOR
\RETURN $X_{k+1} = Q_A Y_{k+1}Q_B^{\T}$.
\end{algorithmic}
\end{algorithm}

\subsection{Convergence}
In this section, we establish that the preconditioning technique accelerates the convergence rate of the fixed-point iteration.

\begin{assumption}\label{assmp}
\begin{itemize}
\item The operator $\hcalL$ is invertible.
\item Assume $\|\hcalL^{-1}\hPi\|<1$, thus the spectral radius satisfies $\rho(\hcalL^{-1}\hPi)<1$, ensuring that equation \eqref{eq:Hsylv} admits a unique solution. 
\end{itemize}
\end{assumption}
Let $Y^{*}$ be the solution of equation \eqref{eq:Hsylv}, satisfying:
\begin{equation}
Y^{*} = \hcalL^{-1}(\hPi(Y^{*})) + \hcalL^{-1}(\hC),\label{eq:Ystart}
\end{equation}
and define the error as $E_{k} = Y_k - Y^{*}$.
The following theorem shows that the convergence of the fixed-point iteration.
\begin{theorem}\label{thm:FP}
Under Assumption \ref{assmp}, the fixed-point iteration \eqref{eq:Yk}
 converges linearly to the unique solution $Y^*$ of equation \eqref{eq:Hsylv}.
 \end{theorem}
\begin{proof}

According to \eqref{eq:Ystart} and $E_k = Y_k - Y^{*}$, we have
\begin{align*}
E_{k+1} &= Y_{k+1} - Y^{*} \\
& = \hcalL^{-1}(\hPi(Y_k)) + \hcalL^{-1}(\hC) - \hcalL^{-1}(\hPi(Y^{*})) - \hcalL^{-1}(\hC) \\
& = \hcalL^{-1}(\hPi(Y_k - Y^{*}))\\
& = \hcalL^{-1}\hPi(E_k).
\end{align*}
Thus, the error satisfies the recurrence
\begin{equation*}
E_{k+1} = \hcalL^{-1}\hPi(E_k) = (\hcalL^{-1}\hPi)^{k+1}(E_0).
\end{equation*}
Since $\|\hcalL^{-1}\hPi\|<1$, 
\begin{equation}
\|E_{k+1}\| \leq \|\hcalL^{-1}\hPi\|^{k+1} \|E_0\|. \label{eq:Ebound-fix}
\end{equation}
Consequently, the fixed-point iteration \eqref{eq:Yk} converges linearly to the unique solution $Y^*$
of equation \eqref{eq:Hsylv} for any initial guess.
\end{proof}

Now, we turn our attention to the convergence of the preconditioned iteration scheme $Y_{k+1} = Y_k - \widetilde{\calP}^{-1}(\calR(Y_k))$, where the residual $\calR(Y_k) = \hcalL(Y_k) - \hPi(Y_k) - \hC$ and $\widetilde{\calP}^{-1}$ is defined by \eqref{eq:secondNeu}. 
\begin{lemma}\label{lem:P}
Let $\widetilde{\calP}^{-1}$ be the approximate preconditioner defined by the first-order truncation of the $(\hcalL - \hPi)^{-1}$, that is,
\begin{equation*}
\widetilde{\calP}^{-1} = \hcalL^{-1} + \hcalL^{-1}\hPi\hcalL^{-1}.
\end{equation*}
For any matrix $Z\in\bbR^{n\times n}$, the following holds:
\begin{equation*}
\widetilde{\calP}^{-1}(\hcalL - \hPi)(Z) = Z - (\hcalL^{-1}\hPi)^2(Z).
\end{equation*}
\end{lemma}
\begin{proof}
\begin{align*}
\widetilde{\calP}^{-1}(\hcalL - \hPi)(Z) &= (\hcalL^{-1} + \hcalL^{-1}\hPi\hcalL^{-1})(\hcalL - \hPi)(Z)\\
& = \hcalL^{-1}\hcalL(Z) - \hcalL^{-1}\hPi(Z) + \hcalL^{-1}\hPi\hcalL^{-1}\hcalL(Z) - \hcalL^{-1}\hPi\hcalL^{-1}\hPi(Z)\\
& = Z - \hcalL^{-1}\hPi(Z) + \hcalL^{-1}\hPi(Z) - \hcalL^{-1}\hPi\hcalL^{-1}\hPi(Z)\\
& = Z - (\hcalL^{-1}\hPi)^2(Z).
\end{align*}
\end{proof}
Based on Lemma \ref{lem:P}, we have the following theorem.
\begin{theorem}\label{thm:PaAA}
Under the Assumption \ref{assmp}, the preconditioned iteration scheme $Y_{k+1} = Y_k - \widetilde{\calP}^{-1}(\calR(Y_k))$, in the Step 12 of  \Cref{alg:pred_AAA}, converges linearly to the unique solution $Y^{*}$ of equation \eqref{eq:Hsylv}.
\end{theorem}
\begin{proof} 
Since the residual $\calR(Y^{*}) = 0$ and $\hcalL(Y^*) = \hPi(Y^{*}) + \hC$, we have 
\begin{equation*}
\calR(Y_k) = \hcalL(E_k) - \hPi(E_k) = (\hcalL - \hPi)(E_k).
\end{equation*}
Substituting into the preconditioned iteration scheme
\begin{align*}
E_{k+1} &= Y_{k+1} - Y^{*}\\
& = Y_{k} - Y^{*} - \widetilde{\calP}^{-1}((\hcalL - \hPi)(E_{k}))\\
& = E_{k} - \widetilde{\calP}^{-1}((\hcalL - \hPi)(E_{k})).
\end{align*}
Applying Lemma \ref{lem:P},
\begin{align*}
E_{k+1} &= E_k - (E_k - (\hcalL^{-1}\hPi)^2(E_k))\\
& = (\hcalL^{-1}\hPi)^2(E_k).
\end{align*}
Thus, we have 
\begin{equation*}
E_{k+1} = (\hcalL^{-1}\hPi)^2(E_k) = [(\hcalL^{-1}\hPi)^2]^{k+1}(E_0).
\end{equation*}
Analogous to Theorem \ref{thm:FP}, 
we obtain
\begin{equation}
\|E_{k+1}\| \leq \|\hcalL^{-1}\hPi\|^{2(k+1)}\|E_0\|. \label{eq:Ebound-pred}
\end{equation}
\end{proof}

From the two bounds above, it can be observed that the error decay rate of formula \eqref{eq:Ebound-pred} is faster than that of the fixed-point iteration in formula \eqref{eq:Ebound-fix}. This demonstrates that the preconditioning technique can accelerate the convergence of the fixed-point iteration. And extensive numerical experiments in the next section confirm the effectiveness of the P-aAA algorithm.

\section{Experimental results}
\label{sec:experiments}
This section presents the numerical experiments conducted in this study. All data used are either publicly available or can be generated directly. Our objectives are twofold. First, for small- to medium-scale problems, we show that the P-aAA algorithm achieves faster convergence and lower computational time than AA and aAA. We also compare its performance with an iterative method based on the Neumann series expansion (Neumann) \cite{jarlebring2018} and with the basic fixed-point iteration, and we assess their accuracy and computational cost against a direct solver based on the Sherman–Morrison–Woodbury (SMW) formula \cite{hao2021}.

Second, for large-scale generalized Sylvester matrix equations, we combine the proposed algorithms with Krylov subspace methods to solve these equations. We compare the performance of two state-of-the-art projection algorithms: the GLEK method \cite{shank2016} and the low-rank commutating Krylov subspace algorithm (LRCK) \cite{jarlebring2018}.
More precisely:

FP: The fixed-point iteration based on the iteration scheme $\calL(X_{k+1}) =  \Pi(X_k) +  C$ for solving
the generalized Sylvester matrix equation. The stopping criterion is $10^{-9}$.

Neumann: An iterative algorithm derived from Neumann series expansion (Neumann). The MATLAB code for Neumann is  available at the web page\footnote{https://sites.google.com/site/palittadavide/home/software}.

SMW: A direct solver based on the Sherman–Morrison–Woodbury formula for solving the dense matrix equation
\eqref{eq:gensylv}, with
$M_i$ and $N_i$ are low-rank matrices. The MATLAB code for SMW method is shown in \cite{hao2021}.

GLEK: Extended Krylov subspace method for solving large-scale generalized Lyapunov matrix equation. The inner and outer tolerances are set $10^{-2}$ and $10^{-9}$. The GLEK code is available at the web page of Simoncini\footnote{http://www.dm.unibo.it/˜simoncin/software.html}.

LRCK: a low-rank commuting Krylov subspace method combined with a Neumann series--based iteration for solving generalized Lyapunov and Sylvester matrix equations. The inner tolerance is set $10^{-14}$, as suggested in the original code. We set $10^{-9}$ as the outer tolerance.  The MATLAB code for LRCK method is available at the web page\footnote{https://sites.google.com/site/palittadavide/home/software}. 

All the experiments are performed in MATLAB 2023b on a laptop with 16GB RAM.

\subsection{Application to small-to-medium-scale generalized Lyapunov and Sylvester equations}
This section presents selected computational results comparing the performance of AA, aAA, and P-aAA algorithms applied to the small-to-medium-scale generalized Lyapunov and Sylvester equation. We set $m_{Max} = 2$
in the AA algorithm. For the AA, aAA, and P-aAA algorithms, we choose $AA_{start} = 10$ in Example $4.2$
and $AA_{start} = 5$ in all other examples. The truncated SVD of $\Delta F_k$ is computed by MATLAB \texttt{svds}
and $r$ is obtained by $\sigma_r/\sigma_{1} \ge 0.1$ and $\sigma_{r+1}/\sigma_{1} < 0.1$, where $\sigma_i$ are singular value of $\Delta F_k$ in decreasing order.
The spectral radius $\rho(\calL^{-1}\Pi)$ for examples $4.1$-$4.5$ is reported in Tables \ref{tab:example4.1-4.4}
and \ref{tab:example4.5}.

\texttt{EXAMPLE $4.1$} The matrix $A$ of the first problem is borrowed from  \texttt{HEAT1} in \cite{benner2013}
and
\begin{equation*}
N_1=
\begin{bmatrix}
8 & -1 & -0.05 &  &  \\
-1 & 8 & -1 & \ddots &  \\
-0.05 & -1 & \ddots & \ddots & -0.05 \\
 & \ddots & \ddots & \ddots & -1 \\
 &  & -0.05 & -1 & 8
\end{bmatrix}\in\mathbb{R}^{900\times 900},
\end{equation*}
$B = A$, $M_1 = N_1$, 
and $C = FF^{\T}$ with $F = \texttt{randn}(900,5)$. 

Figure \ref{fig:example4.1-4.4} demonstrates that the AA, aAA, and P-aAA algorithms all require fewer iterations than the FP, GLEK, and Neumann methods, P-aAA takes the lowest iteration count overall. In terms of computation time, the FP, GLEK, and Neumann methods require significantly longer to converge to the specified tolerance compared to AA, aAA, and P-aAA. Notably, P-aAA achieves the shortest computation time among all the tested methods. These experimental results are summarized in Table \ref{tab:example4.1-4.4}.

\texttt{EXAMPLE $4.2$}
Our second problem is a synthetic problem from \cite{kressner2015}.
The coefficient matrices are given by
\begin{equation*}
A = -\frac{R + R^{\T}}{2} - \frac{n}{8}I_{400},\quad N_1 = -\left (2(S + S^{\T}) + \frac{3n}{4}I_{400}\right)/10^{2},
\quad N_1 = M_1,
\end{equation*}
where the matrices $R = 1.6\times\texttt{rand}(400,400)$ and $S=2.4\times\texttt{rand}(400,400)$ are random matrices generated with the MATLAB function \texttt{rand}. We set $C = FF^{\T}$ and $F = \texttt{rand}(400,10)$.

As shown in Table \ref{tab:example4.1-4.4}, the computational time of P-aAA is significantly lower than that of the other algorithms. Furthermore, Figure \ref{fig:example4.1-4.4} demonstrates that P-aAA requires the fewest number of iterations to converge to the specified tolerance. The computational times of AA and aAA are comparable to each other, and both are substantially shorter than those of GLEK, Neumann, and FP. Meanwhile, GLEK, Neumann, and FP all require exactly the same number of iterations.

\texttt{EXAMPLE $4.3$} The third example is from the the bilinear system described in \cite{lin2009}, yields the following
generalized Lyapunov equation
\begin{equation*}
A X + XA^{\T} + \g^2 \mathop{\sum}\limits_{i=1}^2 N_i XN_i^{\T}
= CC^{\T},
\end{equation*}
where $\g=1/3$, $A = \texttt{tridiag}(2,-5,2)\in\bbR^{1000\times 1000}$, $N_1 = \texttt{tridiag}(3,0,-3)\in\bbR^{1000\times 1000}$, and
$N_2 = -N_1 + I \in\bbR^{1000\times 1000}$, $I\in\bbR^{1000\times 1000}$ is an identity matrix. 
We consider
$C\in\bbR^{1000\times 2}$ being a normalized random matrix.

As shown in Figure \ref{fig:example4.1-4.4} and Table \ref{tab:example4.1-4.4}, the P-aAA algorithm performs favorably compared to all other methods in terms of both the number of iterations and total computation time. Specifically, P-aAA requires significantly fewer iterations than the other algorithms and achieves the shortest overall computation time. In contrast, both AA and aAA do not outperform GLEK, Neumann, or FP in either iteration count or computational efficiency.

\texttt{EXAMPLE $4.4$} This problem arises from the second-order Carleman bilinearization of a nonlinear control system that appears in RC circuit simulations \cite{bai2006}. It reads
\begin{equation*}
AX + XA^{\T} + N_1XN_1^{\T} = C,
\end{equation*}
where $A$, $N_1$, $C$ are of size $n = n_0 + n_0^2$ and are 
given by 
\begin{align*}
A &= \begin{bmatrix}
A_1 &  A_2 \\
0     &  A_1 \otimes I + I \otimes A_1
\end{bmatrix}, \quad
N_1 &= \begin{bmatrix}
0 & 0 \\
b \otimes I + I \otimes b & 0
\end{bmatrix}, \quad
C &= - \begin{bmatrix}
b \\ 0
\end{bmatrix}
\begin{bmatrix}
b^{\T} & 0
\end{bmatrix},
\end{align*}
where $A_1\in\bbR^{n_0\times n_0}$,
$A_2\in\bbR^{n_0\times n_0^2}$ and $b\in\bbR^{n_0}$.
We refer to \cite{bai2006} for the 
explicit construction of the various blocks and
we set $n_0 = 30$ $(n = 930)$. 

In this example, the spectral radius $\rho(\calL^{-1}\Pi) > 1$.
As shown in Figure \ref{fig:example4.1-4.4}, the AA, GLEK, Neumann, and SMW algorithms all fail to converge, whereas both aAA and P-aAA remain convergent. Moreover, P-aAA converges to a residual on the order of $10^{-9}$, while the residual of aAA reaches the order of $10^{-8}$ after $200$ iterations. In addition, P-aAA achieves shorter computational time and requires fewer iterations than aAA.

\begin{figure}[h!]
\centering{
\subfloat[Example $4.1$]{\includegraphics[scale=0.3]{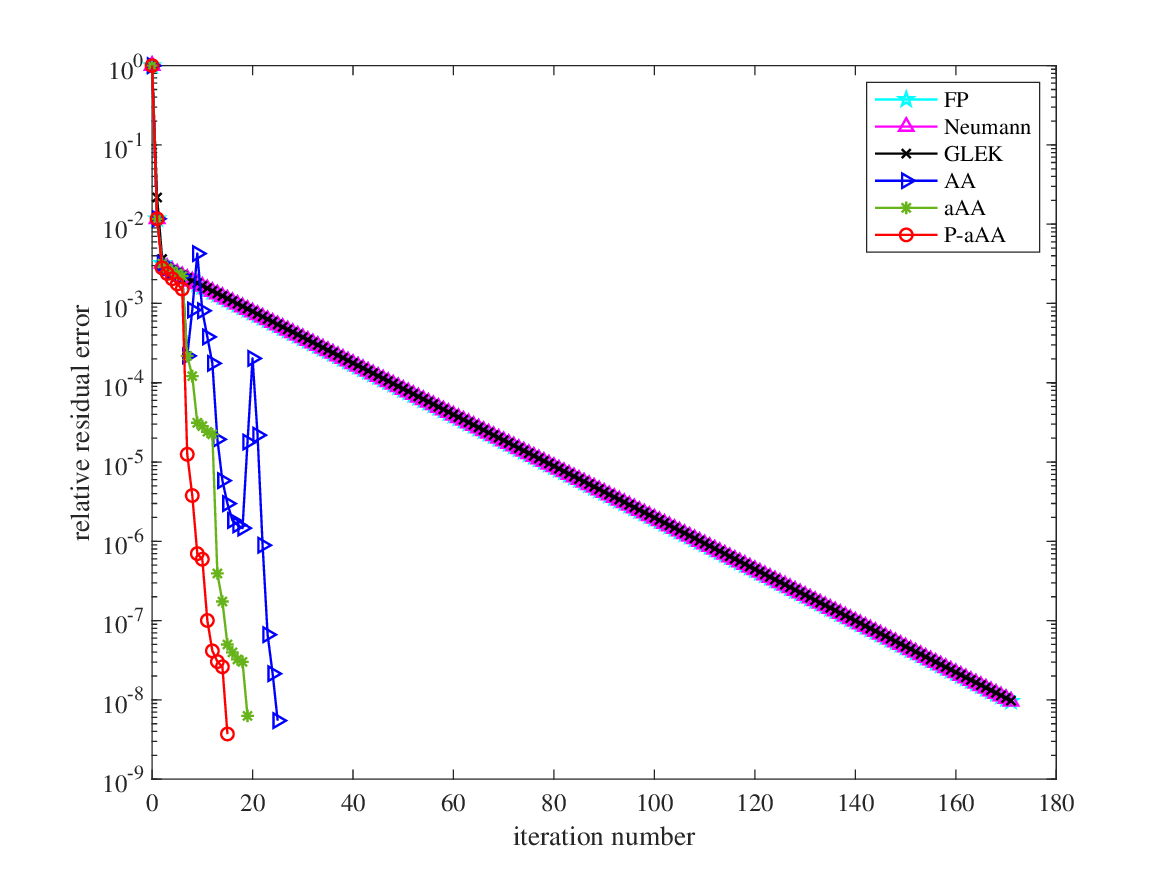}}
\subfloat[Example $4.2$]{\includegraphics[scale=0.3]{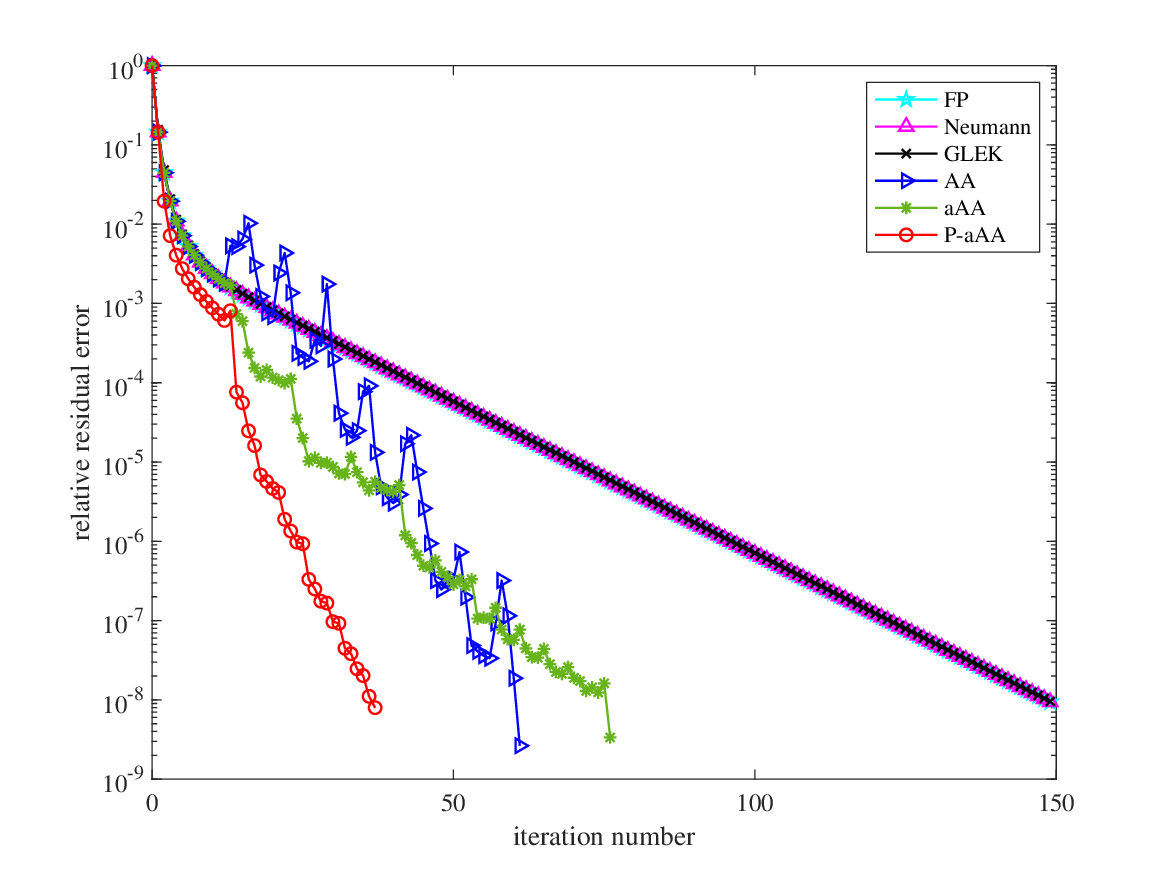}}\\
\subfloat[Example $4.3$]{\includegraphics[scale=0.3]{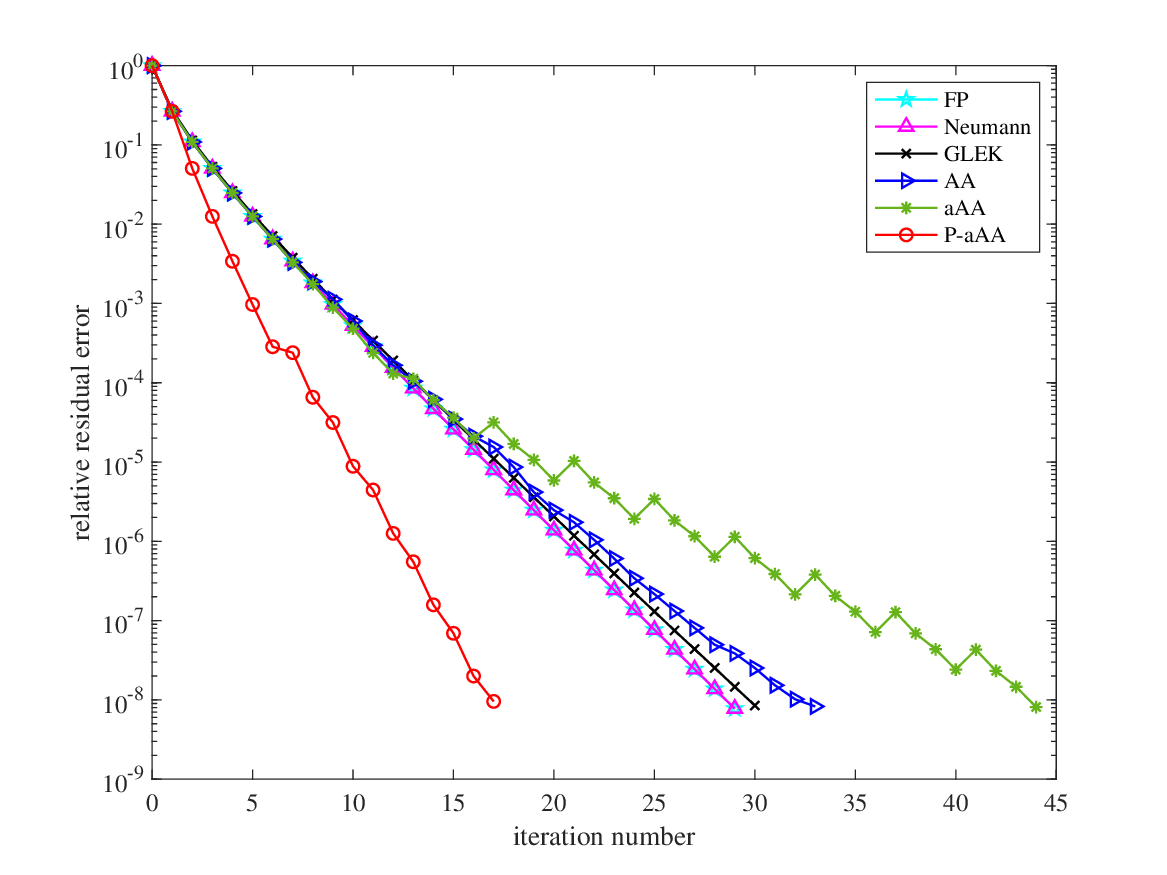}}
\subfloat[Example $4.4$]{\includegraphics[scale=0.3]{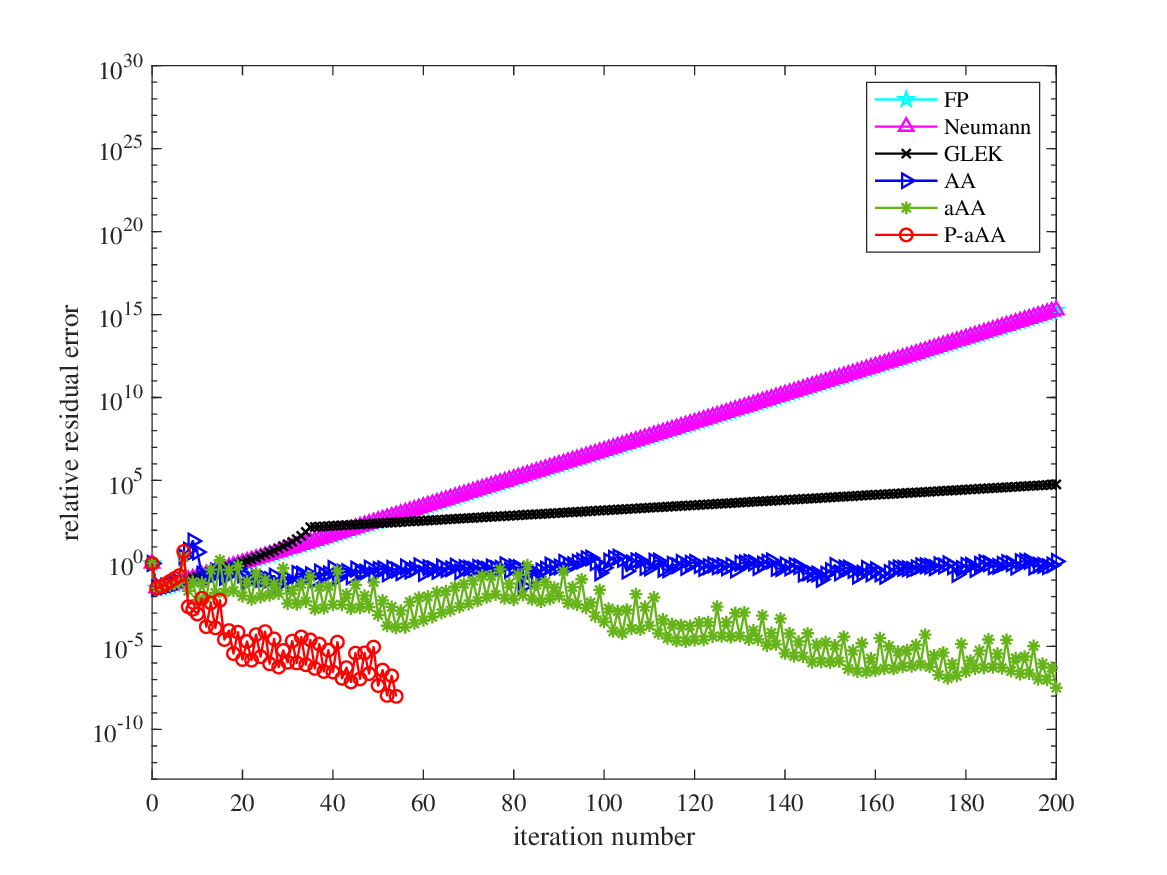}}
}
\caption{The iteration numbers and relative residual error of the proposed methods, FP, Neumann, and GLEK on Examples $4.1$-$4.4$.}\label{fig:example4.1-4.4}
\end{figure}

\begin{table}[htp]
\centering
\caption{The time requirement (in seconds) of the proposed methods, GLEK, Neumann, and FP and the spectral radius $\rho(\calL^{-1}\Pi)$ in Examples $4.1$-$4.4$.}
\label{tab:example4.1-4.4} 
\resizebox{\textwidth}{!}{
\begin{tabular}{|c|ccccccc|}\hline
& $\rho(\calL^{-1}\Pi)$ &  AA &  aAA &   P-aAA &  GLEK & Neumann & FP  \\ \hline
Example $4.1$& $0.93$ &
$3.40 $ & $2.36$ &$\bf 2.21$ & $6.67\times 10^1 $ & $2.47\times 10^1$ & $1.45\times 10^1$ \\ \hline
Example $4.2$& $0.92$ &
$0.95$ & $0.92$ & $\bf 0.63$ & $2.94\times 10^2$ & $1.30$& $1.40$ \\ \hline
Example $4.3$& $0.57$ &
$6.32$ & $6.50$& $\bf 3.56$ & $1.05\times 10^2$ &$3.90$ & $3.76$  \\ \hline
Example $4.4$& $1.22$ &
--- & $2.01\times 10^1$ & $\bf 7.34$ & --- & --- & --- \\ \hline
\end{tabular}}
\end{table}

\texttt{EXAMPLE $4.5$} This example considers a generalized Lyapunov equation of the form
\begin{equation*}
AX + XA^{\T} + N_1 X N_1^{\T} = CC^{\T},
\end{equation*}
where $N_1 = \alpha U_1U_1^{\T}$ and $U_1\in\bbR^{500\times \ell}$. 
We compare the computational time and the relative error between the proposed method and the SMW method, with the column dimension $\ell$ of $U_1$ set to $40$, $50$, and $60$, respectively. We set the parameters $\alpha$ to keep the spectral radius $\rho(\calL^{-1}\Pi) < 1$ and $\alpha$ 
are collected in Table \ref{tab:example4.5}.

From Figure \ref{fig:example4.5}, it can be observed that P-aAA, AA, and aAA achieve smaller relative residuals than SMW. For $\ell = 40$ and $50$, P-aAA requires fewer iterations than both AA and aAA. When $\ell = 60$,
AA attains the smallest number of iterations, P-aAA requires slightly more, and aAA has the largest iteration count.
As shown in Table \ref{tab:example4.5}, the computational time of SMW increases with the number of columns in $N_1$,
whereas P-aAA achieves the shortest computational time among all the algorithms.

\begin{figure}[h!]
\centering{
\subfloat[$\ell = 40$]{\includegraphics[scale=0.21]{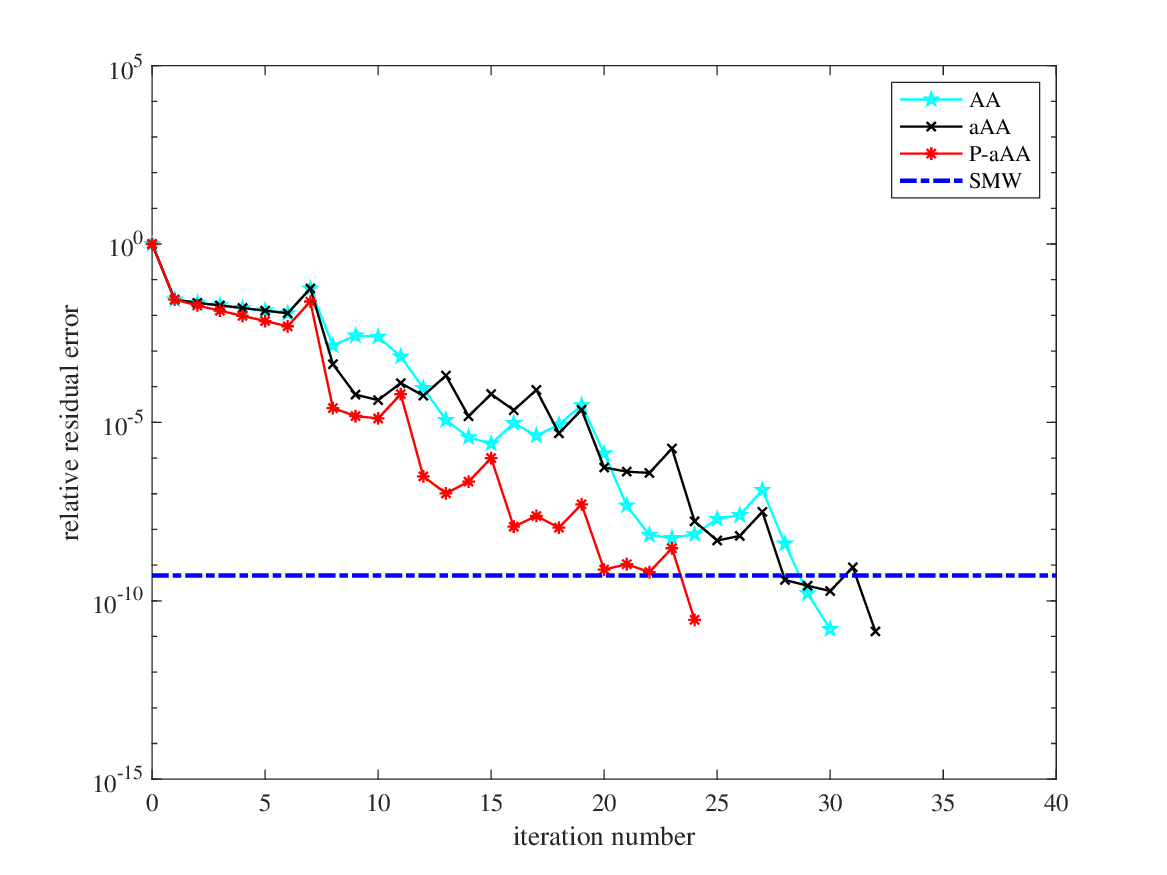}}
\subfloat[$\ell = 50$]{\includegraphics[scale=0.21]{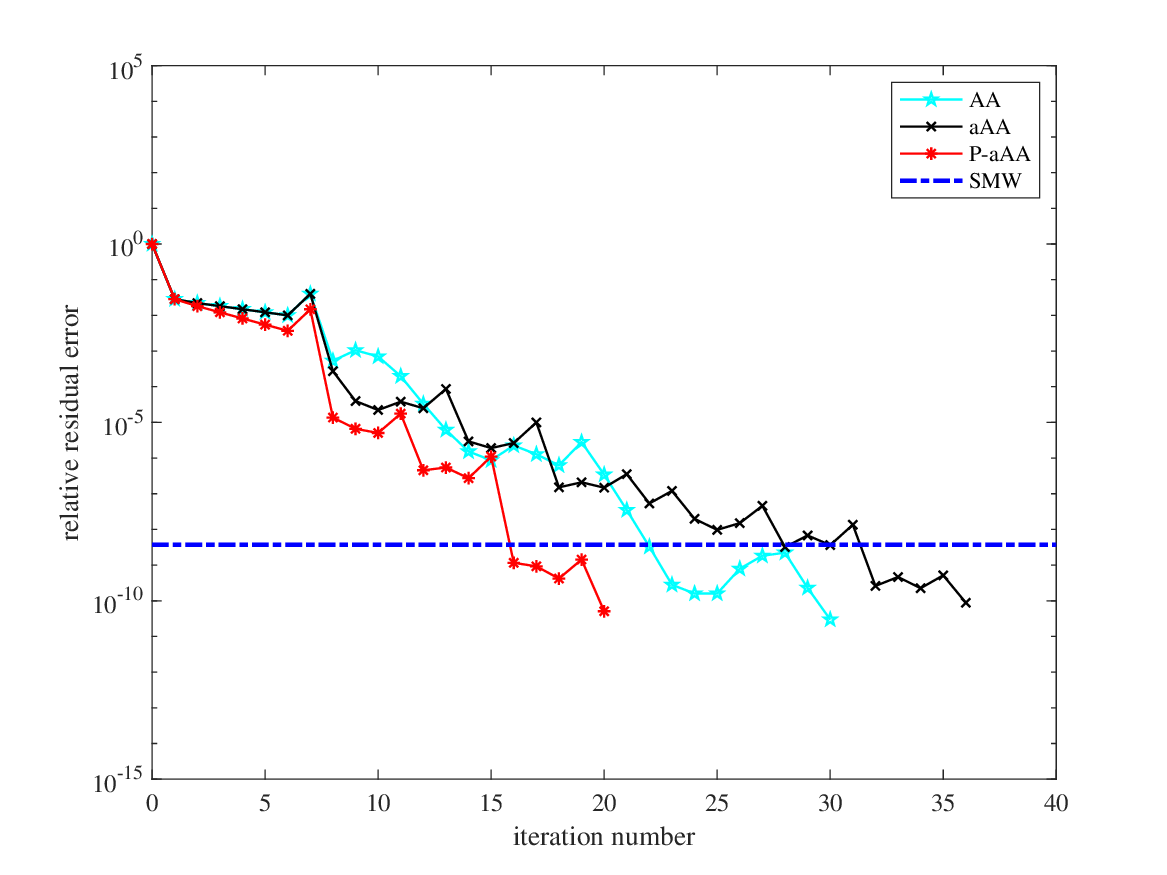}}
\subfloat[$\ell = 60$]{\includegraphics[scale=0.21]{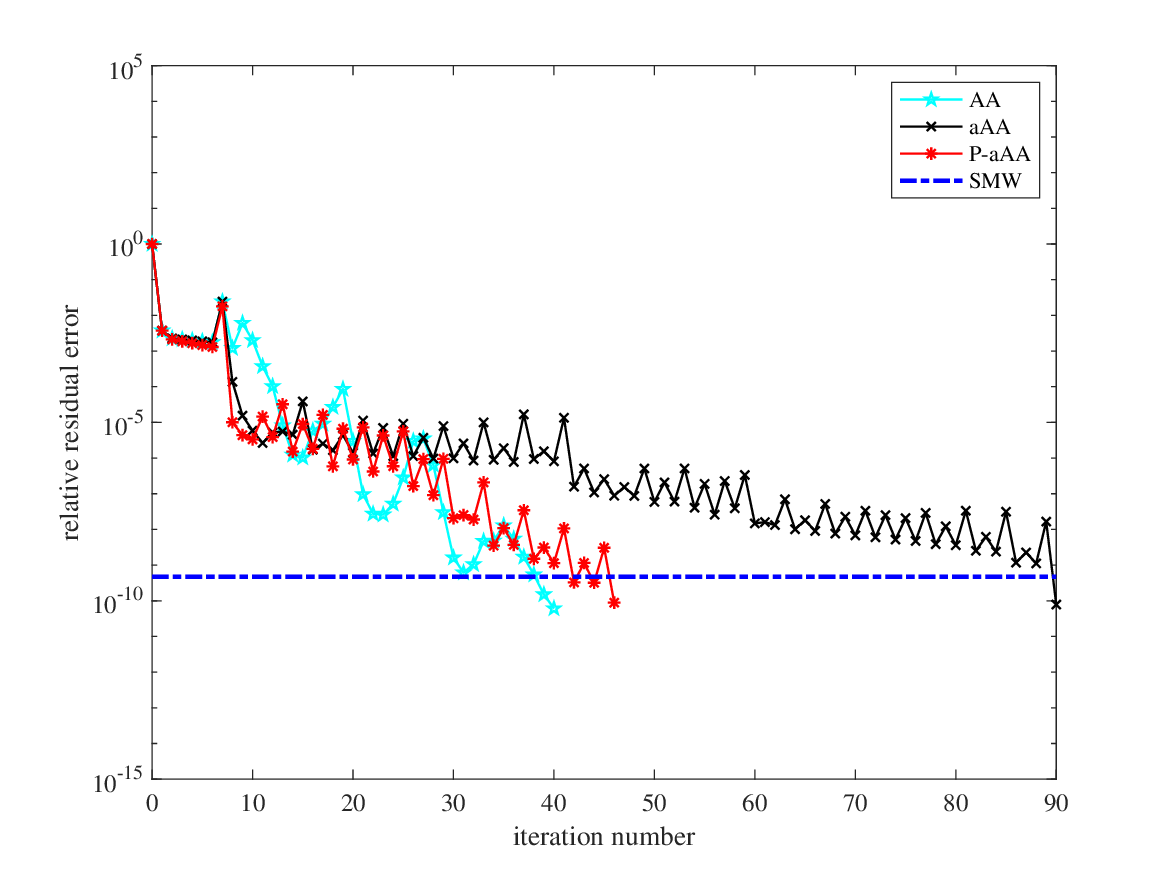}}
}
\caption{The iteration numbers and relative residual error of the proposed methods, FP, Neumann, and GLEK on Example $4.5$.}\label{fig:example4.5}
\end{figure}

\begin{table}[h!]
\centering
\caption{The time requirement (in seconds) of the proposed methods and the SMW method when the columns of $U\in\mathbb{R}^{500\times \ell}$ are chosen by $40$, $50$, and $60$ in Example $4.5$.}
\label{tab:example4.5} 
\resizebox{0.8\textwidth}{!}{
\begin{tabular}{|c|cccccc|}\hline
 & $\alpha$  & $\rho(\calL^{-1}\Pi)$ &  SMW  &   AA & aAA& P-aAA  \\ \hline
$\ell = 40$  & $1.70\times 10^{-4}$  & $0.85$
 & $2.83$ & $0.82$& $0.63$&$\bf 0.61$ \\ \hline
$\ell = 50$  & $1.40\times 10^{-4}$ &  
$0.82$ & $7.94$ & $0.88$ & $0.74$ & $\bf 0.54$ \\ \hline
$\ell = 60$ & $1.25\times 10^{-4}$ & $0.93$ & $28.5$ &$1.48$& $1.80$ & $\bf 1.20$
\\ \hline
\end{tabular}}
\end{table}

\subsection{Application to large scale generalized Lyapunov and Sylvester equations}
In this section, we combine the LRCK method with AA, aAA, and P-aAA to solve large-scale problems.
In the LRCK algorithm, we replace the Neumann method with AA, aAA, or P-aAA to solve the projected generalized Sylvester matrix equation.We choose $AA_{start} = 5$ for AA and $AA_{start} = 1$ for aAA, and P-aAA algorithms.
The inner and outer tolerances of the proposed methods are both set to $10^{-9}$.

\texttt{EXAMPLE $4.6$} This example considers the large-scale case of Example $4.3$. We set 
$\g = 1/6,1/4,1/3$ and $n=50000$. We compare the computational time and the number of iterations of the proposed algorithms with those of the GLEK and LRCK methods.

From Table \ref{tab:example4.6},
we observe that as $\g$ increases, the computational time of GLEK gradually increases, which is consistent with the numerical results reported in [1]. In contrast, P-aAA significantly reduces the computational time compared with GLEK. Moreover, both AA and aAA also achieve faster computation than the GLEK method. On the other hand, the numbers of iterations of AA, aAA, and P-aAA are nearly the same as that of the LRCK method, while P-aAA attains the shortest computational time among all the algorithms. This is confirmed in 
Figure \ref{fig:example4.6} and Table \ref{tab:example4.6}.

\begin{figure}[h!]
\centering{
\subfloat[$\g =  1/6$]{\includegraphics[scale=0.21]{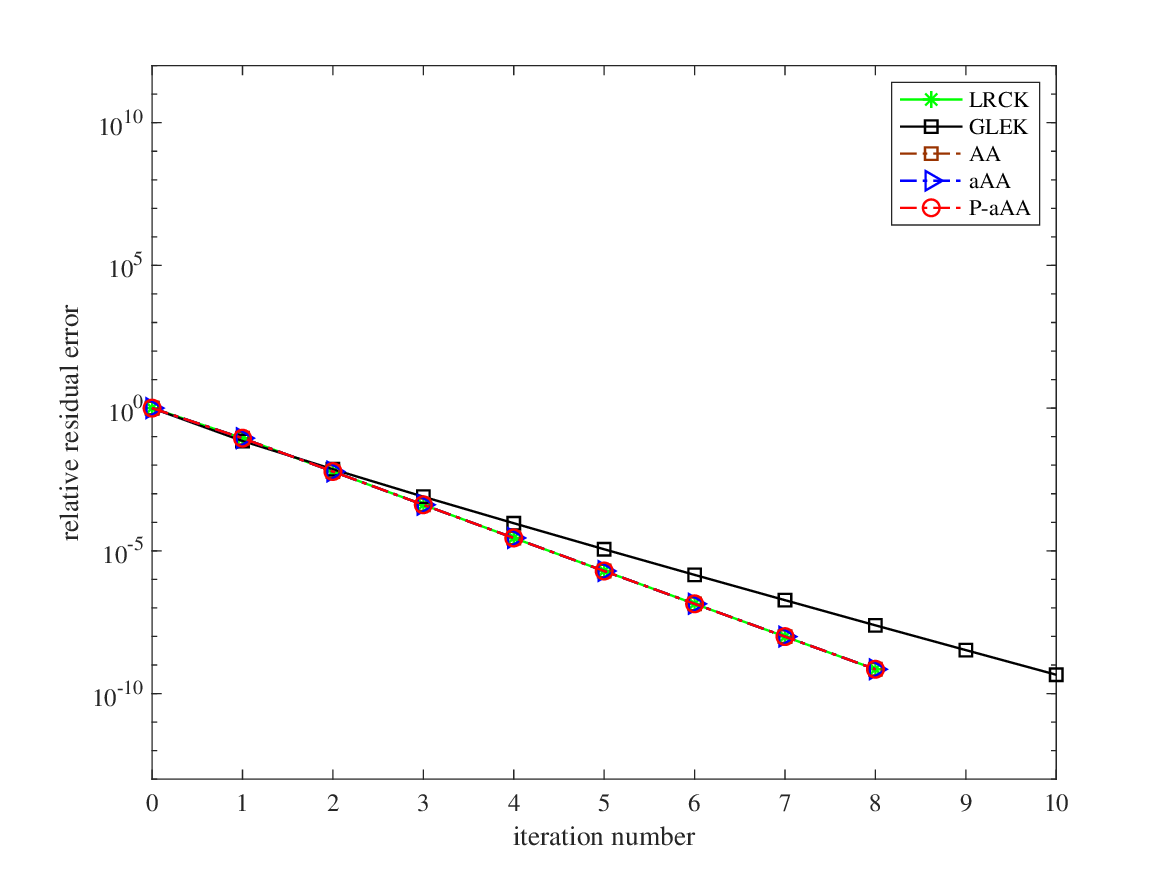}}
\subfloat[$\g = 1/4$]{\includegraphics[scale=0.21]{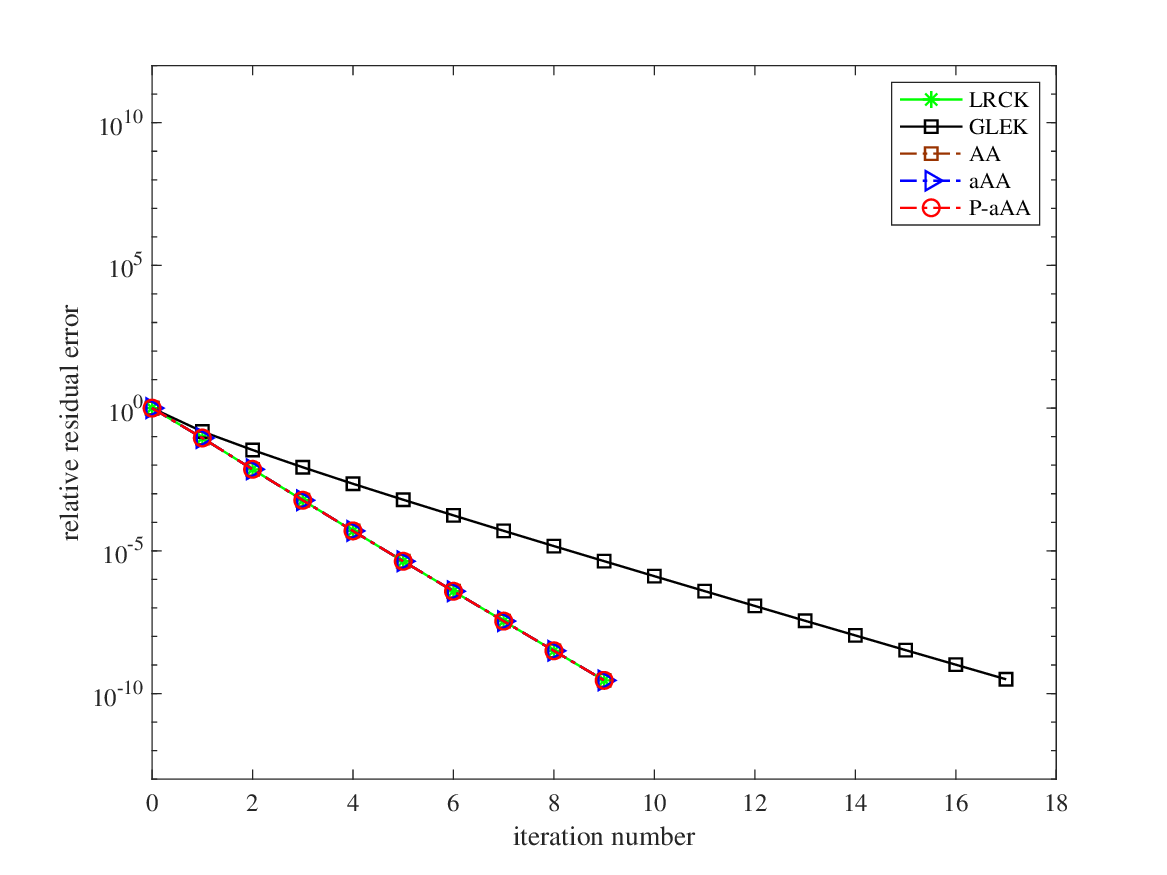}}
\subfloat[$\g = 1/3$]{\includegraphics[scale=0.21]{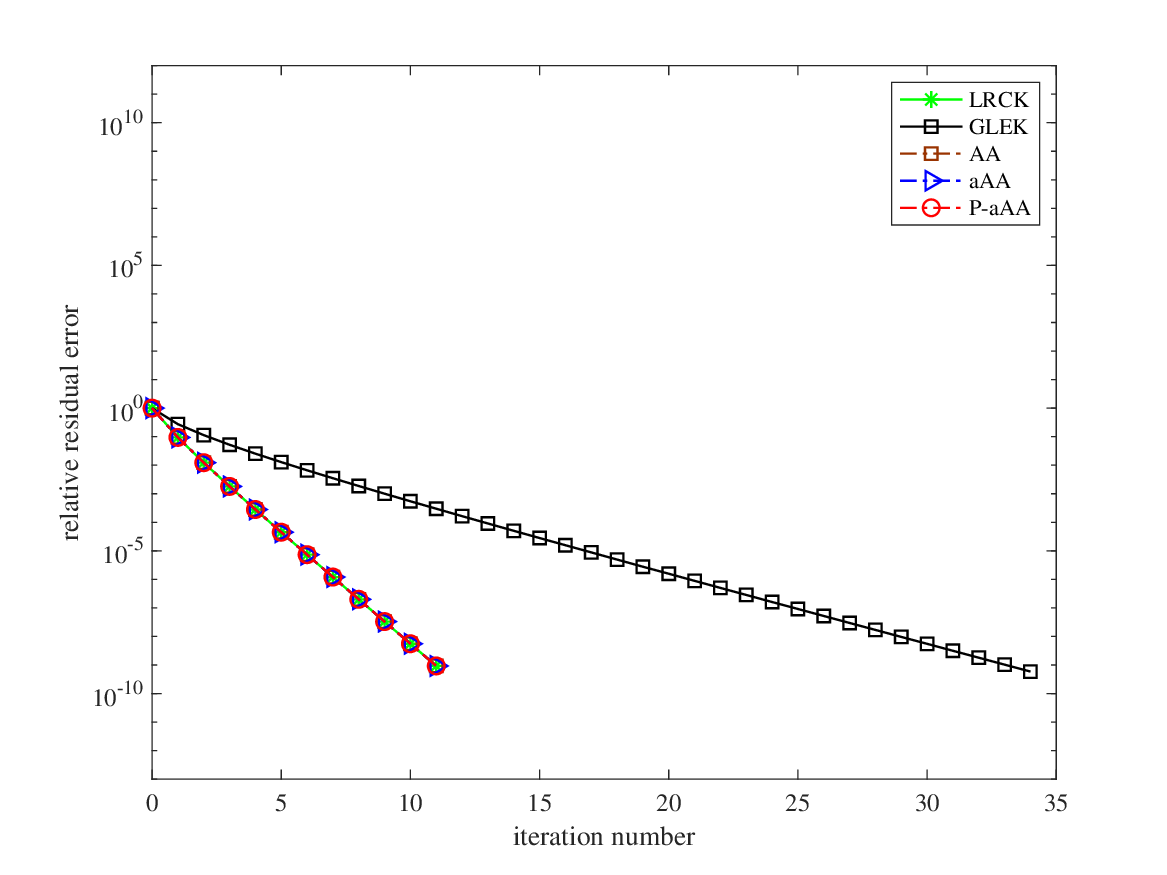}}
}
\caption{The iteration numbers and relative residual error of the proposed methods, GLEK, and LRCK on Examples $4.6$.}\label{fig:example4.6}
\end{figure}

\begin{table}[h!]
\centering
\caption{The time requirement (in seconds) of the proposed methods, GLEK, and LRCK in Example $4.6$.}
\label{tab:example4.6}  
\resizebox{1.0\textwidth}{!}{
\begin{tabular}{|c|ccccc|}\hline
$\g$  &GLEK&  LRCK(Neumann)&   LRCK(AA) & LRCK(aAA)& LRCK(P-aAA)  \\ \hline
$1/6$  & 
$7.83 $ & $1.52$ & $1.47$& $1.41$&$\bf 1.38$ \\ \hline
$1/4$ &  
$21.0 $ & $1.89$ & $2.17$ & $1.74$ & $\bf 1.64$ \\ \hline
$1/3$ & $72.2$ & $2.88$ &$3.55$& $2.63$ & $\bf 2.37$
\\ \hline
\end{tabular}}
\end{table}

\texttt{EXAMPLE $4.7$} The final example is derived from the finite-difference discretization of a partial differential equation. In particular, we consider the inhomogeneous Helmholtz equation
\begin{equation}
\begin{cases}
-\Delta u(x,y) + \kappa(x,y)u(x,y) = f(x,y), & (x,y)\in[0,1]\times\mathbb{R},\\[1mm]
u(0,y)=u(1,y)=0,\\
u(x,y+1)=u(x,y).
\end{cases}
\label{eq:hel}
\end{equation}
From \cite{jarlebring2018}
a finite-difference discretization of \eqref{eq:hel}, using $n$ grid points with $n$ a multiple of four, yields the generalized Sylvester equation
\begin{equation}
AX + XB^{\T} + N_1 X N_1^{\T} = CC^{\T}.
\label{eq:helSy}
\end{equation}
Here, $n = 10000$, 
\[
B = -\texttt{tridiag}(1,-2,1)/h^{2}, \quad
h = 1/(n-1),
\]
\[
A = B - (e_1,e_n)(e_n,e_1)^{\T}/h^{2},
\]
and
\[
N_1 =
\begin{bmatrix}
O_{n/2} & O_{n/2} \\
O_{n/2} & I_{n/2}
\end{bmatrix} \in \mathbb{R}^{n\times n}, \quad
C = [c_1,\dots,c_n]^{\T},
\]
with
\[
c_i =
\begin{cases}
10, & i \in [n/4,n/2],\\
0, & \text{otherwise}.
\end{cases}
\]

Since the matrix $A$ is singular, the equation \eqref{eq:helSy} can be equivalently rewritten as
\begin{equation}
(A+I)X + XB^{\T} + N_1 X N_1^{\T} - X = CC^{\T}.
\label{eq:shiftSy}
\end{equation}

We apply the LRCK(AA), LRCK(aAA), LRCK(P-aAA), and LRCK(Neumann) algorithms to solve \eqref{eq:shiftSy} and compare their computational times and iteration counts. As shown in the \Cref{tab:example4.7}, the LRCK(P-aAA) algorithm achieves the shortest computational time, and its number of iterations is identical to that of LRCK(AA) and LRCK(aAA). Although the LRCK(Neumann) algorithm requires only one fewer iteration than these three methods, their computational times are significantly shorter than that of LRCK(Neumann).

\begin{table}[h!]
\centering
\caption{The time requirement (in seconds) and iteration numbers of the proposed methods and LRCK in Example $4.7$.}
\label{tab:example4.7}  
\resizebox{0.8\textwidth}{!}{
\begin{tabular}{|c|cccc|}\hline
 &LRCK(Neumann)&   LRCK(AA) & LRCK(aAA) & LRCK(P-aAA)  \\ \hline
 Time & 
$8.51 $ & $3.82$& $3.94$&$\bf 3.67$ \\ \hline
 Iteration number&  
$\bf 43$  & $44$ & $44$ & $44$ \\ \hline
\end{tabular}}
\end{table}

\section{Conclusions}
\label{sec:conclusions}

This paper introduced the Preconditioned Alternating Anderson Acceleration (P-aAA) method for solving generalized Sylvester matrix equations. By combining a matrix-oriented Anderson acceleration scheme with a preconditioner constructed from a first-order Neumann series approximation, P-aAA achieves accelerated convergence without increasing computational complexity. We further extended the method to large-scale problems through a Krylov P-aAA framework.

Extensive numerical experiments demonstrated that P-aAA consistently outperforms AA, aAA, Neumann-type iterations, and fixed-point methods, achieving substantial reductions in both computational time and iteration count. The results indicate that P-aAA is an effective and scalable solver for generalized Sylvester matrix equations with general coefficient matrices and challenging spectral properties.

\section*{Acknowledgments}
We thank Professor Kuan Xu of the University of Science and Technology of China for the very helpful discussion in various phases of this project. 
\bibliographystyle{siamplain}
\bibliography{references}
\end{document}